\documentclass{amsart}

\usepackage{amsmath}
\usepackage{amsfonts}
\usepackage{amsthm} 
\usepackage{amssymb}

\usepackage{graphicx} 

\usepackage{color} 

\usepackage{booktabs} 

\newtheorem{thm}{Theorem}[section]
\newtheorem{prop}[thm]{Proposition}
\newtheorem{cor}[thm]{Corollary}
\newtheorem{lem}[thm]{Lemma}

\newtheorem{defn}[thm]{Definition}

\numberwithin{equation}{section} 

\makeatletter 
\@namedef{subjclassname@2010}{2010 Mathematics Subject Classification} 
\makeatother 

\begin{document}

\title[On oscillatory integrals with degenerate phase functions]
{On oscillatory integrals associated to phase functions with degenerate singular points} 

\author{Toshio NAGANO and Naoya MIYAZAKI} 

\address{Department of Liberal Arts, Faculty of Science and Technology, Tokyo University of Science, 2641, Yamazaki, Noda, Chiba 278-8510, JAPAN} 
\email{tonagan@rs.tus.ac.jp}

\address{Department of Mathematics, Faculty of Economics, Keio University, Yokohama, 223-8521, JAPAN} 
\email{miyazaki@a6.keio.jp}

\thanks{The first author was supported by Tokyo University of Science Graduate School doctoral program scholarship 
and an exemption of the cost of equipment from 2016 to 2018 
and would like to thank to Professor Minoru Ito for 
giving me an opportunity of studies and preparing the environment.} 

\keywords{oscillatory integral, asymptotic expansion, stationary phase method} 

\subjclass[2010]{Primary 42B20 ; Secondary 41A60, 58K05} 

\date {September 20, 2020} 

\begin{abstract} 
In this note, by using the result in one variable, 
we obtain asymptotic expansions of oscillatory integrals for certain multivariable phase functions with {\bf degenerate} singular points. 
Moreover by using this result, we have asymptotic expansions of oscillatory integrals with phase function of type $A_{k}$, $E_6$, $E_8$-function germs. 
\end{abstract} 

\maketitle 

\section{Introduction} 

We study asymptotic expansions of oscillatory integrals: 
\begin{align} 
Os\text{-}\int_{\mathbb{R}^{n}} e^{i \lambda \phi(x)} a (x) dx 
:= \lim_{\varepsilon \to +0} \int_{\mathbb{R}^{n}} 
e^{i \lambda \phi(x)} a (x) \chi (\varepsilon x) dx, 
\notag 
\end{align} 
as a positive real parameter $\lambda \to \infty$, where $\chi \in \mathcal{S}(\mathbb{R}^{n})$ with $\chi(0) = 1$ and $0 < \varepsilon < 1$. 

As to a method of asymptotic expansions of oscillatory integrals for phase functions with a {\bf non-degenerate} singular point, 
the stationary phase method is known (\cite{Hormander01}, \cite{Hormander02}, \cite{Duistermaat01}, \cite{Fujiwara1}). 

We are interested in asymptotic expansions of oscillatory integrals for phase functions with a {\bf degenerate} singular point expressed by real polynomials $\sum_{j=1}^{n} \pm x^{m_{j}}$, 
and for amplitude functions belonging to the class $\mathcal{A}^{\tau}_{\delta}(\mathbb{R}^{n})$ (Definision \ref{A_tau_delta}). 

In \cite{Nagano-Miyazaki02}, we obtained the following result in one variable, for a phase function $\phi(x) = \pm x^{m}$ where $m \in \mathbb{N}$ (Theorem \ref{th02}): for any $N > m$, as $\lambda \to \infty$, 
\begin{align} 
Os\text{-}\int_{-\infty}^{\infty} e^{\pm i \lambda x^{m}} a (x) dx 
= \sum_{k=0}^{N-m-1} c_{k}^{\pm} \frac{a^{(k)}(0)}{k!} \lambda ^{-\frac{k+1}{m}} + O\left( \lambda ^{-\frac{N-m+1}{m}} \right), 
\notag 
\end{align} 
where $c_{k}^{\pm}$ is given by \eqref{c_k^pm}. 

In this note, by using this result, we obtained the following result in multivariable, for a phase function $\phi(x) = \sum_{j=1}^{n} \pm_{j} x^{m_{j}}$ where $m_{j} \in \mathbb{N}$ such that $m_{1} \geq \cdots \geq m_{n} \geq 2$ and $\pm_{j}$ stands for ``$+$" or ``$-$" determined by $j$ (Theorem \ref{multivariable}): for any $N_{1} \in \mathbb{N}$ such that $N_{1} > m_{1}$, as $\lambda \to \infty$, 
\begin{align} 
&Os\text{-}\int_{\mathbb{R}^{n}} e^{i \lambda \sum_{j=1}^{n} \pm_{j} x^{m_{j}}} a (x) dx 
= \sum_{\alpha \in \Omega} c_{\alpha} \lambda^{-\sum_{j=1}^{n} \frac{\alpha_{j}+1}{m_{j}}} + O \left( \lambda^{-\frac{N_{1}-m_{1}+1}{m_{1}}} \right), 
\notag 
\end{align} 
where $\Omega$ and $c_{\alpha}$ are given by \eqref{Omega_k} and \eqref{c_alpha}. 

Moreover by using our result, 
in the cases of the phase function $\phi$ is type $A_{k}$, $E_6$, $E_8$-function germ where $k \in \mathbb{N}$ (\cite{Arnold01}, \cite{AGV}): 
\begin{align} 
&A_{k} : \pm x_{1}^{k+1}+x_{2}^{2}+x_{3}^{2}, & 
&E_{6} : \pm x_{1}^{4}+x_{2}^{3}+x_{3}^{2}, & 
&E_{8} : x_{1}^{5}+x_{2}^{3} + x_{3}^{2}, 
\notag 
\end{align} 
we can obtain asymptotic expansions of oscillatory integrals (Corollary \ref{A_E_6_E_8}): 
for any $N \in \mathbb{N}$ such that $N > k+1$, $N > 4$ and $N > 5$ respectively, as $\lambda \to \infty$, 
\begin{align} 
Os\text{-} \int_{\mathbb{R}^{3}} e^{i\lambda (\pm x_{1}^{k+1}+x_{2}^{2}+x_{3}^{2})} a(x) dx 
&= \sum_{\alpha \in \Omega_{1}} c_{\alpha} \lambda^{-\frac{\alpha_{1}+1}{k+1}-\frac{\alpha_{2}+1}{2}-\frac{\alpha_{3}+1}{2}} + O \left( \lambda^{-\frac{N-k}{k+1}} \right), \notag \\ 
Os\text{-} \int_{\mathbb{R}^{3}} e^{i\lambda (\pm x_{1}^{4}+x_{2}^{3}+x_{3}^{2})} a(x) dx 
&= \sum_{\alpha \in \Omega_{2}} c_{\alpha} \lambda^{-\frac{\alpha_{1}+1}{4}-\frac{\alpha_{2}+1}{3}-\frac{\alpha_{3}+1}{2}} + O \left( \lambda^{-\frac{N-3}{4}} \right), \notag \\ 
Os\text{-} \int_{\mathbb{R}^{3}} e^{i\lambda (x_{1}^{5}+x_{2}^{3}+x_{3}^{2})} a(x) dx 
&= \sum_{\alpha \in \Omega_{3}} c_{\alpha} \lambda^{-\frac{\alpha_{1}+1}{5}-\frac{\alpha_{2}+1}{3}-\frac{\alpha_{3}+1}{2}} + O \left( \lambda^{-\frac{N-4}{5}} \right). 
\notag 
\end{align} 

As to a study related to this result, there is a \cite{Duistermaat02}. 

To the end of \S1, we remark 
notation which will be used in this letter: 

$\alpha = (\alpha_{1},\dots,\alpha_{n}) \in \mathbb{Z}_{\geq 0}^{n}$ 
is a multi-index with a length 
$| \alpha | = \alpha_{1} + \cdots + \alpha_{n}$, 
and then, we use 
$x^{\alpha} = x_{1}^{\alpha_{1}} \cdots x_{n}^{\alpha_{n}}$, 
$\alpha! = \alpha_{1}! \cdots \alpha_{n}!$, 
$\partial_{x}^{\alpha} 
= \partial_{x_{1}}^{\alpha_{1}} \cdots \partial_{x_{n}}^{\alpha_{n}}$, 
$D_{x}^{\alpha} = D_{x_{1}}^{\alpha_{1}} \cdots D_{x_{n}}^{\alpha_{n}}$ 
and $\tilde{D}_{x}^{\alpha} = \tilde{D}_{x_{1}}^{\alpha_{1}} \cdots \tilde{D}_{x_{n}}^{\alpha_{n}}$, 
where 
$\partial_{x_{j}} = \frac{\partial}{\partial x_{j}}$, 
$D_{x_{j}} = i^{-1} \partial_{x_{j}}$ 
and $\tilde{D}_{x_{j}} = \lambda^{-1} D_{x_{j}}$ 
for $x = (x_{1}, \dots, x_{n})$ and $\lambda > 0$. 

$C^{\infty}(\mathbb{R}^{n})$ is 
the set of complex-valued functions of class $C^{\infty}$ on $\mathbb{R}^{n}$. 
$C^{\infty}_{0}(\mathbb{R}^{n})$ is 
the set of all $f \in C^{\infty}(\mathbb{R}^{n})$ with compact support. 
$\mathcal{S}(\mathbb{R}^{n})$ is the Schwartz space of rapidly decreasing functions of class $C^{\infty}$ on $\mathbb{R}^{n}$, 
that is, the Fr\'{e}chet space of all $f \in C^{\infty}(\mathbb{R}^{n})$ 
such that $\max_{k+|\alpha| \leq m} \sup_{x \in \mathbb{R}^{n}} \langle x \rangle^{k} | \partial_{x}^{\alpha} f (x) | < \infty$ for any $m \in \mathbb{Z}_{\geq 0}$, where $\langle x \rangle := (1+|x|^{2})^{1/2}$. 

$[x]$ is the Gauss' symbol for $x \in \mathbb{R}$, that is, $[x] \in \mathbb{Z}$ such that $x-1 < [x] \leq x$. 

$O$ means the Landau's symbol, that is, 
$f(x) = O(g(x))~(x \to a)$ if $|f(x)/g(x)|$ 
is bounded as $x \to a$ for functions $f$ and $g$, where $a \in \mathbb{R} \cup \{ \pm \infty \}$. 

$\delta_{ij}$ is the Kronecker's delta, that is, $\delta_{ii} = 1$, and $\delta_{ij} = 0$ if $i \ne j$. 

$\tau^{+} := \max \{ \tau,0 \}$ for $\tau \in \mathbb{R}$. 

\section{Preliminaries}

In this section, 
we shall recall asymptotic expansions of oscillatory integrals and the result in one variable. 
First we recall the oscillatory integrals. 
\begin{defn} 
Let $\lambda > 0$ 
and let $\phi$ be a real-valued function of class $C^{\infty}$ on $\mathbb{R}^{n}$ 
and $a \in C^{\infty}(\mathbb{R}^{n})$. 
If there exists the following limit of improper integral: 
\begin{align} 
\tilde{I}_{\phi}[a](\lambda) 
:= Os\text{-}\int_{\mathbb{R}^{n}} e^{i \lambda \phi(x)} a (x) dx 
:= \lim_{\varepsilon \to +0} \int_{\mathbb{R}^{n}} 
e^{i \lambda \phi(x)} a (x) \chi (\varepsilon x) dx \notag 
\end{align} 
independent of $\chi \in \mathcal{S}(\mathbb{R}^{n})$ with $\chi(0) = 1$ and 
$0 < \varepsilon < 1$, 
then we call $\tilde{I}_{\phi}[a](\lambda)$ an oscillatory integral 
where we call $\phi$ 
$($resp. $a$$)$ a phase function $($resp. an amplitude function$)$. 
\end{defn} 

If we suppose certain suitable conditions for $\phi$ and $a$, 
then we can show $\tilde{I}_{\phi}[a](\lambda)$ exists independent of 
$\chi$ and $\varepsilon$ (Theorem \ref{Lax02} (ii)). 
The fundamental properties are the following (cf. \cite{Kumano-go} p.47.): 
\begin{prop} 
\label{chi epsilon} 
Let $\chi \in \mathcal{S}(\mathbb{R}^{n})$ with $\chi(0) = 1$ and $\langle x \rangle := (1+|x|^{2})^{1/2}$. 
Then 
\begin{enumerate} 
\item[(i)] 
$\chi(\varepsilon x) \to 1$ uniformly on any compact set in $\mathbb{R}^{n}$ as $\varepsilon \to +0$. 
\item[(ii)] 
For each multi-index $\alpha \in \mathbb{Z}_{\geq 0}^{n}$, 
there exists a positive constant $C_{\alpha}$ independent of $0 < \varepsilon < 1$ 
such that 
for any $x \in \mathbb{R}^{n}$ 
\begin{align} 
| \partial_{x}^{\alpha} (\chi(\varepsilon x)) | \leq C_{\alpha} \langle x \rangle^{-|\alpha|}. 
\notag 
\end{align} 
\item[(iii)] 
For any multi-index $\alpha \in \mathbb{Z}_{\geq 0}^{n}$ with $\alpha \ne 0$, 
$\partial_{x}^{\alpha} \chi(\varepsilon x) \to 0$ uniformly in $\mathbb{R}^{n}$ as $\varepsilon \to +0$. 
\end{enumerate} 
\end{prop} 

\begin{proof} 
(i) 
Let $K$ be an any compact set in $\mathbb{R}^{n}$. 
Since $\chi \in \mathcal{S}(\mathbb{R}^{n})$, 
put $C_{K} := n \sup_{x \in K} |x| \cdot \sup_{x \in \mathbb{R}^{n}} | \partial_{x} \chi(x) | > 0$, 
then 
by Taylor expansion of $\chi(x)$ at $x=0$, 
for each $\varepsilon' > 0$, 
there exists $\delta := \varepsilon'/C_{K} > 0$, 
for any $\varepsilon > 0$ such that $0 < \varepsilon < \delta$ 
and for any $x \in K$, 
\begin{align} 
| \chi(\varepsilon x) - 1 | 
\leq \sum_{k=1}^{n} \varepsilon |x| \int_{0}^{1} | \partial_{x_{k}} \chi(\theta \varepsilon x) | d\theta 
\leq C_{K} \varepsilon 
< C_{K} \delta 
= C_{K} \cdot \frac{\varepsilon'}{C_{K}} 
= \varepsilon'. 
\notag 
\end{align} 
Therefore $\chi(\varepsilon x) \to 1$ uniformly on any compact set $K$ in $\mathbb{R}^{n}$ as $\varepsilon \to +0$. 

(ii) 
For any multi-index $\alpha \in \mathbb{Z}_{\geq 0}^{n}$, 
\begin{align} 
\partial_{x}^{\alpha} (\chi(\varepsilon x)) 
&= (\partial_{x}^{\alpha} \chi)(y) \big|_{y=\varepsilon x} \cdot \varepsilon^{|\alpha|} 
\label{|x| leq 1} \\ 
&= \{ |y|^{|\alpha|} (\partial_{x}^{\alpha} \chi)(y) \} \big|_{y=\varepsilon x} \cdot |x|^{-|\alpha|}. 
\label{|x| > 1} 
\end{align} 
Here 
since $\chi \in \mathcal{S}(\mathbb{R}^{n})$, 
put $C_{\alpha,k}' := \sup_{x \in \mathbb{R}^{n}} \langle x \rangle^{k} | \partial_{x}^{\alpha} \chi (x)| > 0$ 
for any $k \in \mathbb{Z}_{\geq 0}$. 

If $|x| \leq 1$, 
since $\langle x \rangle \leq \sqrt{2}$, 
then by \eqref{|x| leq 1}, there exists a positive constant $C_{\alpha} := 2^{|\alpha|/2} C_{\alpha,0}'$ independent of $0 < \varepsilon < 1$ such that 
\begin{align} 
| \partial_{x}^{\alpha} (\chi(\varepsilon x)) | 
\leq C_{\alpha,0}' \varepsilon^{|\alpha|} 
\leq C_{\alpha,0}' \{ \sqrt{2} \langle x \rangle^{-1} \}^{|\alpha|} 
= C_{\alpha} \langle x \rangle^{-|\alpha|}. 
\notag 
\end{align} 

If $|x| > 1$, 
since $\langle x \rangle \leq \sqrt{2}|x|$, 
then by \eqref{|x| > 1}, there exists a positive constant $C_{\alpha} := 2^{|\alpha|/2} C_{\alpha,|\alpha|}'$ independent of $0 < \varepsilon < 1$ such that 
\begin{align} 
| \partial_{x}^{\alpha} (\chi(\varepsilon x)) | 
\leq C_{\alpha,|\alpha|}' |x|^{-|\alpha|} 
\leq C_{\alpha,|\alpha|}' \{ \sqrt{2} \langle x \rangle^{-1} \}^{|\alpha|} 
= C_{\alpha} \langle x \rangle^{-|\alpha|}. 
\notag 
\end{align} 

(iii) 
If $\alpha \ne 0$, 
since $|\alpha| \geq 1$, 
by \eqref{|x| leq 1}, 
then for each $\varepsilon' > 0$, 
there exists $\delta := (\varepsilon'/C_{\alpha,0}')^{1/|\alpha|} > 0$, 
for any $\varepsilon > 0$ such that $\varepsilon < \delta$ 
and for any $x \in \mathbb{R}^{n}$, 
\begin{align} 
| \partial_{x}^{\alpha} (\chi(\varepsilon x)) - 0 | 
\leq C_{\alpha,0}' \varepsilon^{|\alpha|} 
< C_{\alpha,0}' \delta^{|\alpha|} 
= C_{\alpha,0}' (\varepsilon'/C_{\alpha,0}') 
= \varepsilon'. 
\notag 
\end{align} 
Therefore $\partial_{x}^{\alpha} \chi(\varepsilon x) \to 0$ uniformly in $\mathbb{R}^{n}$ as $\varepsilon \to +0$. 
\end{proof} 

In the stationary phase method, the typical case is the following: 
\begin{thm} 
\label{stationary example} 
Let $a \in \mathcal{S}(\mathbb{R}^{n})$. 
Then for any $N \in \mathbb{N}$, as $\lambda \to \infty$, 
\begin{align} 
&\int_{\mathbb{R}^{n}} e^{i\lambda \left( \sum_{j=1}^{p} x_{j}^{2} - \sum_{j=p+1}^{n} x_{j}^{2} \right)} a(x) dx \notag \\ 
&= \pi^{\frac{n}{2}} e^{i\frac{\pi}{4} \{ p-(n-p) \}} \sum_{|\beta|=0}^{N-1} \frac{(-1)^{\sum_{j=p+1}^{n} \beta_{j}} i^{|\beta|} \partial_{x}^{2\beta} a(0)}{4^{|\beta|} \beta!} \lambda^{-k-\frac{n}{2}} + O\left( \lambda^{-N-\frac{n}{2}} \right). 
\notag 
\end{align} 
\end{thm}

First we assume as the phase function the following: 
\begin{defn} 
Let $m_{j} \in \mathbb{N}$ such that $\mu_{n} := \min_{j=1,\dots,n} m_{j} \geq 2$. 
Then we define the phase function $\phi_{n}$ as follows: for any $x = (x_{1},\dots,x_{n}) \in \mathbb{R}^{n}$, 
\begin{align} 
\phi_{n}(x) 
:= \sum_{j=1}^{n} \pm_{j} x_{j}^{m_{j}}. 
\label{phase_function_def} 
\end{align} 
\end{defn} 

Next according to the phase function, 
we define the class of amplitude functions as follows 
(cf. \cite{Kumano-go} p.46.): 
\begin{defn} 
\label{A_tau_delta} 
Assume that $\mu_{n} = \min_{j=1,\dots,n} m_{j}$ as above. 
Let $\tau \in \mathbb{R}$ and $-1 \leq \delta < \mu_{n}-1$. 
Then the class of amplitude functions $\mathcal{A}^{\tau}_{\delta}(\mathbb{R}^{n})$ is the set of all $a \in C^{\infty}(\mathbb{R}^{n})$ such that for $l \in \mathbb{Z}_{\geq 0}$, 
\begin{align} 
|a|^{(\tau)}_{l} 
:= \max_{k \leq l} \sup_{x \in \mathbb{R}^{n}} |x|^{-\tau - \delta |\alpha|} |\partial_{x}^{\alpha} a(x)| < \infty. 
\label{A_tau_delta_def03} 
\end{align} 
\end{defn} 

In \cite{Nagano-Miyazaki02}, we obtained the following result in one variable. 
\begin{thm} 
\label{th02} 
Assume that $\lambda > 0$, $m \in \mathbb{N}$ and $a \in \mathcal{A}^{\tau}_{\delta}(\mathbb{R})$. 
Then for any $N \in \mathbb{N}$ such that $N > m$, 
\begin{align} 
Os\text{-}\int_{-\infty}^{\infty} e^{\pm i\lambda x^{m}} a(x) dx 
= \sum_{k=0}^{N-m-1} c_{k}^{\pm} \frac{a^{(k)}(0)}{k!} \lambda ^{-\frac{k+1}{m}} + R_{N}^{\pm}(\lambda) 
\notag 
\end{align} 
and 
\begin{align} 
R_{N}^{\pm}(\lambda) 
:&= \sum_{k=N-m}^{N-1} c_{k}^{\pm} \frac{a^{(k)}(0)}{k!} \lambda ^{-\frac{k+1}{m}} 
+ \frac{1}{N!} Os\text{-} \int_{-\infty}^{\infty} e^{\pm i\lambda x^{m}} x^{N} a^{(N)} (\theta x) dx \notag \\ 
&= O\left( \lambda^{-\frac{N-m+1}{m}} \right)~(\lambda \to \infty) , 
\notag 
\end{align} 
where 
\begin{align} 
c_{k}^{\pm} 
&:= m^{-1} \left\{ e^{\pm i\frac{\pi}{2} \frac{k+1}{m}} + (-1)^{k} e^{\pm (-1)^{m} i\frac{\pi}{2} \frac{k+1}{m}} \right\} \varGamma \left( \frac{k+1}{m} \right), 
\label{c_k^pm} 
\end{align} 
$0 < \theta < 1$ and double signs $\pm$ are in the same order. 
\end{thm} 

\begin{proof} 
We obtained this theorem in \cite{Nagano-Miyazaki02} (Theorem 5.2 (ii)). 
However, we recall the proof for the readers. 

First if $p > q > 0$, then applying Cauchy's integral theorem to a homeomorphic function $e^{iz^{p}} z^{q-1}$ on the domain with 
the anticlockwise oriented boundary $\sum_{j=1}^{4} C_{j}$ defined by $C_{1} := \{ z=r \in \mathbb{C} | 0 < \varepsilon  \leq r \leq R \}$, 
$C_{2} := \{ z=R e^{i\theta} \in \mathbb{C} | 0 \leq \theta \leq \pi /2p \}$, 
$C_{3} := \{ z=-se^{i(\pi /2p)} 
\in \mathbb{C} | -R \leq s \leq -\varepsilon \}$ 
and $C_{4} := \{ z=\varepsilon  e^{-i\tau} 
\in \mathbb{C} | -\pi /2p \leq \tau \leq 0 \}$, and using $\varepsilon \to +0$ and $R \to \infty$ give 
\begin{align} 
\int_{0}^{\infty} e^{\pm ix^{p}} x^{q-1} dx 
= p^{-1} e^{\pm i\frac{\pi}{2} \frac{q}{p}} \varGamma \left( \frac{q}{p} \right). 
\label{I_pq} 
\end{align} 

Secondly if $p > 0$ and $q > 0$, then dividing $\int_{0}^{\infty} e^{\pm i \lambda x^{p}} x^{q-1} a(x) \chi (\varepsilon x) dx$ where $\chi \in \mathcal{S}(\mathbb{R})$ with $\chi(0) = 1$ and $0 < \varepsilon < 1$ by cutoff function $\varphi$, repeating integration by parts by $L := \frac{1}{px^{p-1}} \frac{1}{i\lambda} \frac{d}{dx}$ on the support of $1-\varphi$, making the order of integrand descend to be integrable in the sense of Lebesgue by $L^{*} := - \frac{1}{i\lambda} \frac{d}{dx} \frac{1}{px^{p-1}}$, and appying Lebsgue's convergence theorem with Proposition \ref{chi epsilon} (i) and (iii) give the existence of oscillatory integrals $\tilde{I}_{p,q}^{\pm}[a](\lambda) := Os\text{-}\int_{0}^{\infty} e^{\pm i\lambda x^{p}} x^{q-1} a(x) dx$. 
In particular, if $\lambda=1$ and $a \equiv 1$, then by \eqref{I_pq}, 
\begin{align} 
\tilde{I}_{p,q}^{\pm} 
:= Os\text{-}\int_{0}^{\infty} e^{\pm ix^{p}} x^{q-1} dx 
= p^{-1} e^{\pm i\frac{\pi}{2} \frac{q}{p}} 
\varGamma \left( \frac{q}{p} \right), 
\label{generalized_Fresnel_integral_def} 
\end{align} 
and if $q > p$, then 
\begin{align} 
| \tilde{I}_{p,q}^{\pm}[a](\lambda) | 
= O\left( \lambda^{-\frac{q-p}{p}} \right)~(\lambda \to \infty). 
\label{tilde_I_pq_est} 
\end{align} 

Thirdly dividing $\tilde{I}_{p,1}^{\pm}[a](\lambda) := Os\text{-}\int_{0}^{\infty} e^{\pm ix^{p}} a(x) dx$ into the principal part 
and the remainder term 
by Taylor expansion of $a(x)$ at $x=0$. 
And then applying \eqref{generalized_Fresnel_integral_def} after using change of variable $y=\lambda^{1/p}x$ in the former 
and applying \eqref{tilde_I_pq_est} in the latter give for any $N \in \mathbb{N}$ such that $N \geq p+1$, 
\begin{align} 
Os\text{-}\int_{0}^{\infty} e^{\pm i\lambda x^{p}} a(x) dx 
= \sum_{k=0}^{N-[p]-1} \tilde{I}_{p,k+1}^{\pm} \frac{a^{(k)}(0)}{k!} \lambda ^{-\frac{k+1}{p}} + \tilde{R}_{N}^{\pm}(\lambda), 
\label{tilde_I_p1_expans} 
\end{align} 
where 
\begin{align} 
\tilde{R}_{N}^{\pm}(\lambda) 
:= &\sum_{k=N-[p]}^{N-1} \tilde{I}_{p,k+1}^{\pm} \frac{a^{(k)}(0)}{k!} \lambda ^{-\frac{k+1}{p}} 
+ \frac{1}{N!} Os\text{-} \int_{0}^{\infty} e^{\pm i\lambda x^{p}} x^{N} a^{(N)} (\theta x) dx \notag \\ 
&= O\left( \lambda^{-\frac{N-m+1}{m}} \right)~(\lambda \to \infty). 
\label{R_N_01} 
\end{align} 

Finally set $p=m \in \mathbb{N}$ on \eqref{tilde_I_p1_expans} and \eqref{R_N_01}, 
and if $x<0$, then changing a variable $x = -y$ completes the proof. 
\end{proof} 

\section{Existence and estimates of oscillatory integrals} 

In this section, 
we shall show existence and estimates of oscillatory integrals used in the next section. 
First as to the phase function, we have the following: 
\begin{prop} 
\label{nabla_phi} 
Assume that $\phi_{n}$ is defined by \eqref{phase_function_def}. 
Then there exists a positive constant $C$ such that for any $|x_{j}| \geq 1$ for $j=1,\dots,n$, 
\begin{align} 
C |x|^{\mu_{n}-1} \leq |\nabla \phi_{n} (x)|. 
\notag 
\end{align} 
\end{prop} 

\begin{proof} 
By H\"{o}lder's inequality:$\sum_{j=1}^{n} x_{j}^{2} \cdot 1 
< ( \sum_{j=1}^{n} (x_{j}^{2})^{\mu_{n}-1} )^{1/(\mu_{n}-1)} ( \sum_{j=1}^{n} 1^{q} )^{1/q}$ where $q>1$ and $\{ 1/(\mu_{n}-1)\} + (1/q) = 1$ for $\mu_{n} > 2$. 
\end{proof} 

\begin{prop} 
\label{b_j}
Assume that $\lambda > 0$ and $\phi_{n}$ is defined by \eqref{phase_function_def}. 
Let 
\begin{align} 
L_{n} := \sum_{j=1}^{n} b_{j}(x) \tilde{D}_{x_{j}},~
b_{j}(x) := \frac{\partial_{x_{j}} \phi_{n}(x)}{|\nabla \phi_{n} (x)|^{2}}, 
\label{L def} 
\end{align} 
Then $L_{n}e^{i\lambda \phi_{n} (x)} = e^{i\lambda \phi_{n} (x)}$ and $b_{j} \in C^{\infty}(\mathbb{R}^{n} \setminus \{ 0 \})$, 
and then for each multi-index $\alpha \in \mathbb{Z}_{\geq 0}^{n}$, there exists a positive constant $C_{j,\alpha}$ such that 
for any $|x_{k}| \geq 1$ for $k=1,\dots,n$, 
\begin{align} 
| \partial_{x}^{\alpha} b_{j}(x) | \leq C_{j,\alpha} |x|^{-(\mu_{n}-1)}. 
\label{b_j est} 
\end{align} 
\end{prop} 

\begin{proof} 
By induction on $|\alpha|$. 
If $|\alpha|=0$, by Proposition \ref{nabla_phi}, then \eqref{b_j est} holds. 
If \eqref{b_j est} holds for $|\alpha|-1$, then applying $\partial_{x}^{\alpha}$ to $|\nabla \phi_{n} (x)|^{2} b_{j}(x) = \partial_{x_{j}} \phi_{n}(x)$ and using Leibniz's formula give 
\begin{align} 
&|\nabla \phi_{n} (x)|^{2} \partial_{x}^{\alpha} b_{j}(x) 
= -\sum_{0 < \beta \leq \alpha} \binom{\alpha}{\beta} \partial_{x}^{\beta} |\nabla \phi_{n} (x)|^{2} \partial_{x}^{\alpha - \beta} b_{j}(x) + \partial_{x}^{\alpha} \partial_{x_{j}} \phi_{n} (x). 
\notag 
\end{align} 
Moreover using $\left| \partial_{x}^{\beta} |\nabla \phi_{n} (x)|^{2} \right| \leq C_{\beta} |\nabla \phi_{n} (x)|^{2}$, the hypothesis, $| \partial_{x}^{\alpha} \partial_{x_{j}} \phi_{n}(x) | \leq C_{\alpha} |\nabla \phi_{n} (x)|$ and Proposition \ref{nabla_phi} show \eqref{b_j est} for $|\alpha|$. 
\end{proof} 

By Proposition \ref{b_j}, we obtain the following:
\begin{prop} 
\label{d_j} 
Let 
\begin{align} 
L_{n}^{*} 
:&= - \sum_{j=1}^{n} \tilde{D}_{x_{j}} b_{j}(x) 
\label{L* def0}
\end{align} 
be a formal adjoint operator of $L_{n}$, and denote 
\begin{align} 
L_{n}^{*} 
= \sum_{j=1}^{n} d_{j}(x) \tilde{D}_{x_{j}} + d_{0}(x),~
d_{j}(x) := -b_{j}(x),~
d_{0}(x) := -\sum_{j=1}^{n} \tilde{D}_{x_{j}}(b_{j}(x)). 
\label{d_j def} 
\end{align} 
Then $d_{j} \in C^{\infty}(\mathbb{R}^{n} \setminus \{ 0 \})$ for $j=0,\dots,n$, and 
for each multi-index $\alpha \in \mathbb{Z}_{\geq 0}^{n}$, there exist positive constants $C_{j,\alpha}$ such that 
for any $|x_{k}| \geq 1$ for $k=1,\dots,n$, 
\begin{align} 
&| \partial_{x}^{\alpha} d_{j}(x) | \leq C_{j,\alpha} |x|^{-(\mu_{n}-1)}~\text{for $j=1,\dots,n$}, 
\label{d_j est} \\ 
&| \partial_{x}^{\alpha} d_{0}(x) | \leq C_{0,\alpha} \lambda^{-1} |x|^{-(\mu_{n}-1)}. 
\label{d_0 est} 
\end{align} 
\end{prop} 

\begin{proof} 
The estimates \eqref{b_j est} and \eqref{d_j def} above give \eqref{d_j est} and \eqref{d_0 est}. 
\end{proof} 

By Proposition \ref{d_j}, we obtain the following:
\begin{prop} 
\label{d_l_alpha} 
For any $l \in \mathbb{Z}_{\geq 0}$, denote 
\begin{align} 
L_{n}^{*l} f = \sum_{|\alpha| \leq l} d_{l,\alpha}(x) \tilde{D}_{x}^{\alpha} f 
\label{L*l_def} 
\end{align} 
for $f \in C^{\infty}(\mathbb{R}^{n} \setminus \{ 0 \} )$, where $L_{n}^{*0}$ is an identity operator. 
Then 
$d_{l,\alpha} \in C^{\infty}(\mathbb{R}^{n} \setminus \{ 0 \})$ 
and for each multi-index $\beta \in \mathbb{Z}_{\geq 0}^{n}$, 
there exists a positive constant $C_{l,\beta}$ such that 
for any $|x_{k}| \geq 1$ for $k=1,\dots,n$, 
\begin{align} 
| \partial_{x}^{\beta} d_{l,\alpha}(x) | 
\leq C_{l,\beta} \lambda^{-(l-|\alpha|)} |x|^{-(\mu_{n}-1)l}. 
\label{d_l_alpha_est} 
\end{align} 
\end{prop} 

\begin{proof} 
By induction on $l \in \mathbb{Z}_{\geq 0}$. 
If $l=0$, since $d_{0,0} \equiv 1$, then \eqref{d_l_alpha_est} holds. 
If \eqref{d_l_alpha_est} holds for $l-1$ with $l \geq 1$, 
then for any $l \in \mathbb{N}$ and for any multi-index $\alpha \in \mathbb{Z}_{\geq 0}^{n}$ with $|\alpha| \leq l$, for any $f \in C^{\infty}(\mathbb{R}^{n} \setminus \{ 0 \} )$, 
\begin{align} 
L_{n}^{*l} f
&= \sum_{|\alpha| \leq l-1} \bigg\{ \bigg( \sum_{j=1}^{n} d_{j} \tilde{D}_{x_{j}} d_{l-1,\alpha} + d_{0} d_{l-1,\alpha} \bigg) \tilde{D}_{x}^{\alpha} f + \sum_{j=1}^{n} d_{j} d_{l-1,\alpha} \tilde{D}_{x_{j}} \tilde{D}_{x}^{\alpha} f \bigg\} 
\notag 
\end{align} 
Applying $\partial_{x}^{\beta}$ to the coefficients and using Leibniz's formula gives 
\begin{align} 
&| \partial_{x}^{\beta} d_{l,\alpha}(x) | 
\leq \sum_{|\alpha| \leq l-1} \sum_{\beta' \leq \beta} \binom{\beta}{\beta'} \notag \\ 
&\hspace{0.5cm}\times \bigg\{ \bigg| \partial_{x}^{\beta'} \bigg( \sum_{j=1}^{n} d_{j} \tilde{D}_{x_{j}} d_{l-1,\alpha} + d_{0} d_{l-1,\alpha} \bigg) \bigg| + \bigg| \partial_{x}^{\beta'} \bigg( \sum_{j=1}^{n} d_{j} d_{l-1,\alpha} \lambda^{-1} \bigg) \bigg| \bigg\}. 
\notag 
\end{align} 
Moreover using \eqref{d_j est}, \eqref{d_0 est} and the induction hypothesis shows \eqref{d_l_alpha_est} for $l$. 
\end{proof} 

We also have 
\begin{prop} 
\label{L_star_l} 
\label{multivariable} 
Assume that $\lambda > 0$. 
For any $s=1,\dots,n$, let $\phi_{s}(x') := \sum_{j=1}^{s} \pm_{j} x_{j}^{m_{j}}$, $L_{s} := \sum_{j=1}^{s} b_{j}(x') \tilde{D}_{x_{j}}$ and $L_{s}^{*} := -\sum_{j=1}^{s} \tilde{D}_{x_{j}} b_{j}(x')$ be defined by \eqref{phase_function_def}, \eqref{L def} and \eqref{L* def0}
 for $x=(x',x'') \in \mathbb{R}^{s} \times \mathbb{R}^{n-s}$ respectively. 
And let $a \in \mathcal{A}^{\tau}_{\delta}(\mathbb{R}^{n})$, 
$\varphi_{k} \in C^{\infty}_{0}(\mathbb{R})$ be a cutoff function such that $\varphi_{k} \equiv 1$ on $|x_{k}| \leq 1$ and $\varphi_{k} \equiv 0$ on $|x_{k}| \geq r > 1$, and $\psi_{k} := 1 - \varphi_{k}$ for $k=1,\cdots,n$. 
And for any $s=0,\dots,n$, let 
\begin{align} 
a_{s}(x) 
:= a(x) \prod_{k=s+1}^{n} \varphi_{k}(x_{k}) \prod_{k=1}^{s} \psi_{k}(x_{k}), 
\label{a_s_def} 
\end{align} 
where $\prod_{k=n+1}^{n} \varphi_{k} :=1$ and $\prod_{k=1}^{0} \psi_{k} := 1$. 
And let $\chi \in \mathcal{S}(\mathbb{R}^{n})$ with $\chi(0) = 1$ and $0 < \varepsilon < 1$. 
Then for any $s=1,\dots,n$, the following hold: 
\begin{enumerate} 
\item[(i)] 
For any $j = 1,\dots,s$, for any $l \in \mathbb{Z}_{\geq 0}$ and for any $h = 0,1$, let 
\begin{align} 
f_{j,l,h}(x) 
:&= (b_{j}(x') \tilde{D}_{x_{j}})^{1-h} (e^{i\lambda \phi_{s}(x')}) (-\tilde{D}_{x_{j}})^{h} (b_{j}(x')^{h} L_{s}^{*l} (a_{s}(x) \chi(\varepsilon x))), 
\notag 
\end{align} 
Then 
for any $0 < \varepsilon < 1$ and for any $l \in \mathbb{Z}_{\geq 0}$, the improper integrals 
\begin{align} 
\int_{\mathbb{R}^{n}} e^{i\lambda \phi_{n}(x)} L_{s}^{*l} (a_{s}(x) \chi(\varepsilon x)) dx~\text{and}~
\int_{\mathbb{R}^{n}} e^{i\lambda (\phi_{n}(x) - \phi_{s}(x'))} f_{j,l,h}(x) dx 
\notag 
\end{align} 
are absolutely convergent. 
\item[(ii)] 
For any $j = 1,\dots,s$, for any $(x_{1},\dots,\widehat{x_{j}},\dots,x_{n}) \in \mathbb{R}^{n-1}$, for any $l \in \mathbb{N}$ and for any $0 < \varepsilon < 1$, as $|x_{j}| \to \infty$, 
\begin{align} 
\big| b_{j}(x') (i\lambda)^{-1} e^{i\lambda \phi_{s} (x')} L_{s}^{*(l-1)} (a_{s}(x) \chi(\varepsilon x)) \big| \to 0. 
\notag 
\end{align} 
\item[(iii)] 
For any $l \in \mathbb{N}$ and for any $0 < \varepsilon < 1$, 
\begin{align} 
\int_{\mathbb{R}^{n}} e^{i\lambda \phi_{n}(x)} a_{s}(x) \chi(\varepsilon x) dx 
= \int_{\mathbb{R}^{n}} e^{i\lambda \phi_{n}(x)} L_{s}^{*l} (a_{s}(x) \chi(\varepsilon x)) dx. 
\notag 
\end{align} 
\item[(iv)] 
For each $l \in \mathbb{Z}_{\geq 0}$, 
there exists a positive constant $C_{l}$ independent of $0 < \varepsilon < 1$ such that for any $x \in \mathbb{R}^{n}$, 
\begin{align} 
|L_{s}^{*l}(a_{s}(x) \chi(\varepsilon x))| 
\leq C_{l} \lambda^{-l} \prod_{k=s+1}^{n} ||\varphi_{k}||\hspace{0.05cm}|a|^{(\tau)}_{l} |\psi_{1,s}|^{(l)}_{r} \langle x \rangle^{t}, 
\notag 
\end{align} 
where $t := \tau^{+} -(\mu_{n}-1-\delta^{+})l$, $||\varphi_{k}|| := \sup_{x_{k} \in \mathbb{R}} |\varphi_{k}(x_{k})|$ and 
\begin{align} 
|\psi_{1,s}|^{(l)}_{r} 
:&= 
\max_{|\gamma| \leq l} \sup_{\substack{|x_{k}| < r,\\k=1,\dots,n}} \langle x \rangle^{|\gamma|} \prod_{k=1}^{s} |\partial_{x_{k}}^{\gamma_{k}} \psi_{k}(x_{k})|+1. 
\label{|psi|_{l,2}} 
\end{align} 
\end{enumerate} 
\end{prop} 

\begin{proof} 
First we show (i). 
A function $\chi \in \mathcal{S}(\mathbb{R}^{n})$ gives $L_{s}^{*l}(a_{s}(x) \chi(\varepsilon x)) = O(|x|^{t})$ $(|x| \to \infty)$ when $t < -n$ and $e^{i\lambda (\phi_{n}(x) - \phi_{s}(x'))} f_{j,l,h}(x) = O(|x|^{t})~(|x| \to \infty)$ when $t < -n$. This shows (i). 
Secondly we show (ii). 
A function $\chi \in \mathcal{S}(\mathbb{R}^{n})$ also gives $b_{j}(x') (i\lambda)^{-1} e^{i\lambda \phi_{s} (x')} L_{s}^{*(l-1)} (a_{s}(x) \chi(\varepsilon x)) = O(|x|^{t})~(|x| \to \infty)$ when $t < 0$. This shows (ii). 
Next we show (iii). 
By induction on $l \in \mathbb{N}$. Applying integration by part with (i), (ii) and Fubini's theorem show (iii). 
Finally we show (iv). 
If $x \in \mathrm{supp} L_{s}^{*l}(a_{s}(x) \chi(\varepsilon x))$, since 
\begin{align} 
\mathrm{supp} L_{s}^{*l}(a_{s}(x) \chi(\varepsilon x)) 
\subset \mathrm{supp} a_{s} 
\subset \prod_{k=1}^{s} (\mathbb{R} \setminus (-1,1)) \times \prod_{k=s+1}^{n} (-r,r), 
\notag 
\end{align} 
then $|x_{k}| \geq 1$ for $k=1,\dots,s$ and $|x_{k}| \leq r$ for $k=s+1,\dots,n$. 
And then since $|x'| \geq 1$ and $|x''| \leq (n-s)^{1/2}r$, 
then $|x'|^{-(\mu_{s}-1)l} \leq C_{n,l} \langle x \rangle^{-(\mu_{n}-1)l}$ where $C_{n,l} := \max_{s=0,\dots,n} \{ 2+2(n-s)r^{2} \}^{(\mu_{s}-1)l/2}$. 
Hence \eqref{d_l_alpha_est}, \eqref{A_tau_delta_def03}, \eqref{|psi|_{l,2}} and Proposition \ref{chi epsilon} (ii) give 
\begin{align} 
&|L_{s}^{*l}(a_{s}(x) \chi(\varepsilon x))| 
= \bigg| \prod_{k=s+1}^{n} \varphi_{k}(x_{k}) \sum_{|\alpha| \leq l} d_{l,\alpha}(x') \tilde{D}_{x'}^{\alpha} \bigg( a(x) \prod_{k=1}^{s} \psi_{k}(x_{k}) \chi(\varepsilon x) \bigg) \bigg| \notag \\ 
&\leq \prod_{k=s+1}^{n} |\varphi_{k}(x_{k})| \sum_{|\alpha| \leq l} |d_{l,\alpha}(x')| \lambda^{-|\alpha|} \notag \\ 
&\hspace{0.5cm}\times \sum_{\beta + \gamma + \zeta = \alpha} \frac{\alpha!}{\beta! \gamma! \zeta!} |\partial_{x'}^{\beta} a(x)| \bigg| \partial_{x'}^{\gamma} \prod_{k=1}^{s} \psi_{k}(x_{k}) \bigg| |\partial_{x'}^{\zeta} (\chi(\varepsilon x))| \notag \\ 
&\leq \prod_{k=s+1}^{n} ||\varphi_{k}|| \sum_{|\alpha| \leq l} C_{l,0} \lambda^{-(l-|\alpha|)} |x'|^{-(\mu_{s}-1)l} \lambda^{-|\alpha|} \notag \\ 
&\hspace{0.5cm}\times \sum_{\beta + \gamma + \zeta = \alpha} \frac{\alpha!}{\beta! \gamma! \zeta!} |a|^{(\tau)}_{l} \langle x \rangle^{\tau + \delta |\beta|} \langle x \rangle^{-|\gamma|} \langle x \rangle^{|\gamma|} \prod_{k=1}^{s} |\partial_{x_{k}}^{\gamma_{k}} \psi_{k}(x_{k})| C_{\zeta} \langle x \rangle ^{\delta |\zeta|} \notag \\ 
&\leq \lambda^{-l} \prod_{k=s+1}^{n} ||\varphi_{k}|| \sum_{|\alpha| \leq l} C_{l,0} C_{n,l} \langle x \rangle^{-(\mu_{n}-1)l} \notag \\ 
&\hspace{0.5cm}\times \sum_{\beta + \gamma + \zeta = \alpha} \frac{\alpha!}{\beta! \gamma! \zeta!} |a|^{(\tau)}_{l} \langle x \rangle^{\tau + \delta |\alpha|} |\psi_{1,s}|^{(l)}_{r} C_{\zeta}. \notag 
\end{align} 
This shows (iv). 
\end{proof} 

By Proposition \ref{L_star_l}, 
we obtain the following theorem: 
\begin{thm} 
\label{Lax02} 
Assume that $\lambda > 0$ and $s=0,\dots,n$. 
Let $\phi_{n}$, $a_{s}$ and $L_{s}^{*}$ be defined by \eqref{phase_function_def}, \eqref{a_s_def} and \eqref{L* def0} respectively. 
Then the following hold: 
\begin{enumerate} 
\item[(i)] 
There exists the following oscillatory integral, and the following holds: 
\begin{align} 
\tilde{I}_{\phi_{n}}[a_{0}](\lambda) 
:&= Os\text{-}\int_{\mathbb{R}^{n}} e^{i\lambda \phi_{n}(x)} a_{0}(x) dx 
= \int_{\mathbb{R}^{n}} e^{i\lambda \phi_{n}(x)} a_{0}(x) dx. 
\notag 
\end{align} 
\item[(ii)] 
If $s \ne 0$, 
then there exists the following oscillatory integral, 
and for any $l \in \mathbb{N}$ such that $l \geq l_{n}$, the following holds: 
\begin{align} 
\tilde{I}_{\phi_{n}}[a_{s}](\lambda) 
:&= Os\text{-}\int_{\mathbb{R}^{n}} e^{i\lambda \phi_{n}(x)} a_{s}(x) dx 
= \int_{\mathbb{R}^{n}} e^{i\lambda \phi_{n}(x)} L_{s}^{*l}(a_{s}(x)) dx, 
\notag 
\end{align} 
where $l_{n} := [ (\tau^{+}+n)/(\mu_{n}-1-\delta^{+}) ] +1$. 
And then for each $l \in \mathbb{N}$ such that $l \geq l_{n}$, there exists a positive constant $\tilde{C}_{s,l}$ such that for any $\lambda >  0$, 
\begin{align} 
| \tilde{I}_{\phi_{n}}[a_{s}](\lambda) | 
\leq \tilde{C}_{s,l} \prod_{k=s+1}^{n} ||\varphi_{k}||\hspace{0.05cm}|a|^{(\tau)}_{l_{n}} |\psi_{1,s}|^{(l)}_{r} \lambda^{-l}, 
\notag 
\end{align} 
where $||\varphi_{k}|| := \sup_{x_{k} \in \mathbb{R}} |\varphi_{k}(x_{k})|$ 
and 
$|\psi_{1,s}|^{(l)}_{r}$ is defined by \eqref{|psi|_{l,2}}. 
\end{enumerate} 
\end{thm} 

\begin{proof} 
First we prove (i). 
Applying Lebesgue's convergence theorem to $e^{i\lambda \phi_{n}(x)}$ $a_{0}(x) \chi(\varepsilon x) \in C^{\infty}_{0}(\mathbb{R}^{n})$ with Proposition \ref{chi epsilon} shows (i). 
Next we prove (ii). 
Proposition \ref{L_star_l} (iv) makes the order of integrand descend to be integrable in the sense of Lebesgue: 
For any $l \in \mathbb{N}$ such that $l \geq l_{n}:= [ (\tau^{+}+n)/(\mu_{n}-1-\delta^{+}) ] +1$, 
\begin{align} 
|e^{i\lambda \phi_{n}(x)} L_{s}^{*l}(a_{s}(x) \chi(\varepsilon x))| 
\leq C_{l} \lambda^{-l} \prod_{k=s+1}^{n} ||\varphi_{k}||\hspace{0.05cm}|a|^{(\tau)}_{l} |\psi_{1,s}|^{(l)}_{r} \langle x \rangle^{t} 
=: M(x), 
\notag 
\end{align} 
where $t := \tau^{+} -(\mu_{n}-1-\delta^{+})l < -n$. 
Since $\int_{\mathbb{R}^{n}} M(x) dx$ is absolutely convergent independent of $\chi$ and $\varepsilon$, 
applying Lebsgue's convergence theorem with Proposition \ref{chi epsilon} 
on Proposition \ref{L_star_l} (iii) as $\varepsilon \to +0$, 
completes the proof. 
\end{proof} 

\section{Asymptotic expansions of oscillatory integrals} 

In this section, 
we show asymptotic expansions of oscillatory integrals in multivariable. 
First the following holds. 
\begin{lem} 
\label{split_of_cutoff_function} 
Let $\varphi_{j}$ and $\psi_{j}$ be functions for $j=1,\cdots,n$. 
Then 
\begin{align} 
\prod_{j=1}^{n} (\varphi_{j} + \psi_{j}) 
= \sum_{(j_{1},\dots,j_{n}) \in S_{n}} \sum_{s=0}^{n} \frac{1}{(n-s)! s!} \prod_{k=s+1}^{n} \varphi_{j_{k}} \prod_{k=1}^{s} \psi_{j_{k}}, 
\label{split01} 
\end{align} 
where $S_{n}$ is a symmetric group, $\prod_{k=n+1}^{n} \varphi_{j_{k}} :=1$ and $\prod_{k=1}^{0} \psi_{j_{k}} := 1$. 
\end{lem} 

\begin{proof} 
By induction on $n \in \mathbb{N}$. 
\end{proof} 

Next by Lemma \ref{split_of_cutoff_function}, Theorem \ref{th02} in $\S 2$ and Theorem \ref{Lax02} in $\S 3$, 
we obtain the following theorem: 
\begin{thm} 
\label{multivariable} 
Assume that $\lambda > 0$, $\phi(x):=\sum_{j=1}^{n} \pm_{j} x_{j}^{m_{j}}$ 
where $m_{j} \in \mathbb{N}$ such that $m_{1} \geq \cdots \geq m_{n} \geq 2$ and $\pm_{j}$ stands for ``$+$" or ``$-$" determined by $j$, 
and $a \in \mathcal{A}^{\tau}_{\delta}(\mathbb{R}^{n})$. 
Then for any $N_{1} \in \mathbb{N}$ such that $N_{1} > m_{1}$, as $\lambda \to \infty$, 
\begin{align} 
Os\text{-} \int_{\mathbb{R}^{n}} e^{i\lambda \phi(x)} a(x) dx 
&= \sum_{\alpha \in \Omega} c_{\alpha} \lambda^{-\sum_{j=1}^{n} \frac{\alpha_{j}+1}{m_{j}}} + O \left( \lambda^{-\frac{N_{1}-m_{1}+1}{m_{1}}} \right) 
\label{tilde_I_phi_a_lambda} 
\end{align} 
where 
\begin{align} 
\Omega 
:= \bigg\{ (\alpha_{j}) \in \mathbb{Z}_{\geq 0}^{n} \bigg| \sum_{j=1}^{n} \frac{\alpha_{j}+(1-\delta_{1j})}{m_{j}} < \frac{N_{1}}{m_{1}}-1 \bigg\} 
\label{Omega_k} 
\end{align} 
and 
\begin{align} 
c_{\alpha} 
:&= \prod_{j=1}^{n} c^{\pm_{j}}_{j,\alpha_{j}} \frac{\partial_{x}^{\alpha} a(0)}{\alpha!} 
\label{c_alpha} 
\end{align} 
with 
\begin{align} 
&c^{\pm_{j}}_{j,\alpha_{j}} 
= m_{j}^{-1} \left\{ e^{\pm_{j} i\frac{\pi}{2} \frac{\alpha_{j}+1}{m_{j}}} + (-1)^{\alpha_{j}} e^{\pm_{j} (-1)^{m_{j}} i\frac{\pi}{2} \frac{\alpha_{j}+1}{m_{j}}} \right\} \varGamma \left( \frac{\alpha_{j}+1}{m_{j}} \right) \notag \\ 
&= 
\begin{cases} 
\dfrac{1}{k_{j}} e^{\pm_{j} i \frac{\pi}{2} \frac{2\beta_{j}+1}{2k_{j}}} \varGamma \left( \dfrac{2\beta_{j}+1}{2k_{j}} \right) & \text{for $m_{j}=2k_{j}$, $\alpha_{j}=2\beta_{j}$}, \\ 
\hspace{1.75cm}0 & \text{for $m_{j}=2k_{j}$, $\alpha_{j}=2\beta_{j}+1$}, \\ 
\dfrac{2}{2k_{j}+1} \cos \dfrac{\pi(2\beta_{j}+1)}{2(2k_{j}+1)} \varGamma \left( \dfrac{2\beta_{j}+1}{2k_{j}+1} \right) & \text{for $m_{j}=2k_{j}+1$, $\alpha_{j}=2\beta_{j}$}, \\ 
\dfrac{\pm_{j} 2i}{2k_{j}+1} \sin \dfrac{\pi(\beta_{j}+1)}{2k_{j}+1} \varGamma \left( \dfrac{2\beta_{j}+2}{2k_{j}+1} \right) & \text{for $m_{j}=2k_{j}+1$, $\alpha_{j}=2\beta_{j}+1$}, 
\end{cases} 
\label{c_pm_j_j_alpha_j} 
\end{align} 
where $k_{j} \in \mathbb{N}$ and $\beta_{j} \in \mathbb{Z}_{\geq 0}$ for $j=1,\dots,n$. 
\end{thm} 

\begin{proof} 
First we divide $\int_{\mathbb{R}^{n}} e^{i\lambda \phi(x)} a(x) \chi (\varepsilon x) dx$ into an integral on compact support around the singular point of $\phi$ and integrals on non-compact support: 
After using cutoff function $\varphi_{j}$ for $j=1,\dots,n$ with Lemma \ref{split_of_cutoff_function}, by proper permutations of variables in each term, without lost of generality, we can assume that $\sum_{s=0}^{n} \frac{n!}{(n-s)! s!} \int_{\mathbb{R}^{n}} e^{i\lambda \tilde{\phi}(x)} a_{s}(x) \chi (\varepsilon x) dx$ with \eqref{a_s_def}. 
If $s = 0$, since $a_{0} \in C^{\infty}_{0}(\mathbb{R}^{n})$, then by Theorem \ref{Lax02} (i) and Theorem \ref{th02}, for any $N_{1} \in \mathbb{N}$ such that $N_{1}>m_{1}$, as $\lambda \to \infty$, 
\begin{align} 
&I_{0} 
:= Os\text{-}\int_{\mathbb{R}^{n}} e^{i\lambda \phi(x)} a_{0}(x) dx 
= \int_{\mathbb{R}^{n}} e^{i\lambda \phi(x)} a_{0}(x) 
= \prod_{j=1}^{n} \int_{-\infty}^{\infty} dx_{j} e^{\pm_{j}i\lambda x_{j}^{m_{j}}} a_{0}(x) \notag \\ 
&= \sum_{\alpha_{1}=0}^{N_{1}-m_{1}-1} \prod_{j=2}^{n} \int_{-\infty}^{\infty} dx_{j} e^{\pm_{j}i\lambda x_{j}^{m_{j}}} c_{1,\alpha_{1}}^{\pm_{1}} \lambda^{-\frac{\alpha_{1}+1}{m_{1}}} \frac{\partial_{x_{1}}^{\alpha_{1}}}{\alpha_{1}!} \bigg|_{x_{1}=0} a_{0}(x) + O \left( \lambda^{-\frac{N_{1}-m_{1}+1}{m_{1}}} \right), 
\notag 
\end{align} 
where $c_{1,\alpha_{1}}^{\pm_{1}}$ is given by \eqref{c_pm_j_j_alpha_j}. 
Here for any $j=2,\dots,n$, let $N_{j} \in \mathbb{N}$ such that $N_{j} \geq N_{1}$, 
then $N_{j} > m_{j}$, and for any $\lambda \geq 1$, 
\begin{align} 
\frac{N_{j}}{m_{j}} 
> \frac{N_{1}}{m_{1}}-\sum_{k=2}^{j} \frac{1}{m_{k}} 
\Longleftrightarrow 
\lambda^{-\sum_{k=1}^{j-1} \frac{1}{m_{k}} -\frac{N_{j}-m_{j}+1}{m_{j}}} \leq \lambda^{-\frac{N_{1}-m_{1}+1}{m_{1}}}. 
\label{N_k_condition} 
\end{align} 
Repeating to apply Theorem \ref{th02} with \eqref{N_k_condition} gives inductively
\begin{align} 
&I_{0} 
= \sum_{\alpha_{1}=0}^{N_{1}-m_{1}-1} \cdots \sum_{\alpha_{n}=0}^{N_{n}-m_{n}-1} c_{\alpha} \lambda^{-\sum_{j=1}^{n} \frac{\alpha_{j}+1}{m_{j}}} + O \left( \lambda^{-\frac{N_{1}-m_{1}+1}{m_{1}}} \right)~(\lambda \to \infty), 
\notag 
\end{align} 
where $c_{\alpha}$ is given by \eqref{c_alpha}. 
Let $\Omega$ be defined by \eqref{Omega_k}. 
Then 
\begin{gather} 
\frac{\alpha_{j}+(1-\delta_{1j})}{m_{j}} 
< \sum_{j=1}^{n} \frac{\alpha_{j}+(1-\delta_{1j})}{m_{j}} < \frac{N_{1}}{m_{1}}-1 \leq \frac{N_{j}}{m_{j}}-1 \notag \\ 
\Longrightarrow \alpha_{j} \leq N_{j}-m_{j}-1, 
\notag 
\end{gather} 
for $j=1.\dots,n$, and for any $\lambda \geq 1$, 
\begin{align} 
&\sum_{j=1}^{n} \frac{\alpha_{j}+(1-\delta_{1j})}{m_{j}} < \frac{N_{1}}{m_{1}}-1 
\Longleftrightarrow \lambda^{-\sum_{j=1}^{n} \frac{\alpha_{j}+1}{m_{j}}} > \lambda^{-\frac{N_{1}-m_{1}+1}{m_{1}}}. 
\notag 
\end{align} 
Hence $\lambda \to \infty$, 
\begin{align} 
I_{0} 
:= Os\text{-}\int_{\mathbb{R}^{n}} e^{i\lambda \phi(x)} a_{0}(x) dx 
= \sum_{\alpha \in \Omega} c_{\alpha} \lambda^{-\sum_{j=1}^{n} \frac{\alpha_{j}+1}{m_{j}}} + O \left( \lambda^{-\frac{N_{1}-m_{1}+1}{m_{1}}} \right). 
\notag 
\end{align} 

If $s \ne 0$, since $a_{s}$ has a non-compact support, then Theorem \ref{Lax02} (ii) gives 
\begin{align} 
\sum_{s=1}^{n} \frac{n!}{(n-s)! s!} Os\text{-}\int_{\mathbb{R}^{n}} e^{i\lambda \phi (x)} a_{s}(x) dx 
= O \left( \lambda^{-\frac{N_{1}-m_{1}+1}{m_{1}}} \right)~(\lambda \to \infty). 
\notag 
\end{align} 
\end{proof} 

By Theorem \ref{multivariable}, we obtain the following immdeiately. 
\begin{cor} 
\label{m_j=m} 
Assume that $\lambda > 0$ and $a \in \mathcal{A}^{\tau}_{\delta}(\mathbb{R}^{n})$. 
Then the following hold: 
\begin{enumerate} 
\item[(i)] 
Let $m \in \mathbb{N}$ such that $m \geq 2$. 
Then for any $N \in \mathbb{N}$ such that $N > m$, as $\lambda \to \infty$, 
\begin{align*} 
Os\text{-} \int_{\mathbb{R}^{n}} e^{i\lambda \sum_{j=1}^{n} \pm_{j} x_{j}^{m}} a(x) dx 
&= \sum_{\alpha \in \Omega} c_{\alpha} \lambda^{-\frac{|\alpha|+n}{m}} + O \left( \lambda^{-\frac{N-m+1}{m}} \right), 
\notag 
\end{align*} 
where $\Omega$ and $c_{\alpha}$ are defined by \eqref{Omega_k} when $N_{1}=N$, and \eqref{c_alpha} with 
\begin{align*} 
c^{\pm_{j}}_{j,\alpha_{j}} 
&= m^{-1} \left\{ e^{\pm_{j} i\frac{\pi}{2} \frac{\alpha_{j}+1}{m}} + (-1)^{\alpha_{j}} e^{\pm_{j} (-1)^{m} i\frac{\pi}{2} \frac{\alpha_{j}+1}{m}} \right\} \varGamma \left( \frac{\alpha_{j}+1}{m} \right). 
\end{align*} 
\item[(ii)] 
For any $N \in \mathbb{N}$ such that $N > 2$, as $\lambda \to \infty$, 
\begin{align*} 
&\int_{\mathbb{R}^{n}} e^{i\lambda \left( \sum_{j=1}^{p} x_{j}^{2} - \sum_{j=p+1}^{n} x_{j}^{2} \right)} a(x) dx \notag \\ 
&= \pi^{\frac{n}{2}} e^{i\frac{\pi}{4} \{ p-(n-p) \}} \sum_{\beta \in \Omega'} \frac{(-1)^{\sum_{j=p+1}^{n}} i^{|\beta|} \partial_{x}^{2\beta} a(0)}{4^{|\beta|} \beta!} \lambda^{-|\beta|-\frac{n}{2}} + O \left( \lambda^{-\frac{N}{2}+\frac{1}{2}} \right),  
\notag 
\end{align*} 
where $\Omega' = \{ (\beta_{j}) \in \mathbb{Z}_{\geq 0}^{n} |~|\beta| < (N-n-1)/2 \}$. 
\end{enumerate} 
\end{cor} 

\begin{proof} 
(i) Set $m_{j}=m \geq 2$ for $j=1,\dots,n$ on Theorem \ref{multivariable}. 
(ii) Set $m=2$, $\pm_{j}=+$ for $j=1,\dots,p$ and $\pm_{j}=-$ for $j=p+1,\dots,n$ on (i), and apply \eqref{c_pm_j_j_alpha_j}.   
\end{proof} 

By Theorem \ref{stationary example} in \S 2 and Corollary \ref{m_j=m}, 
we can consider that Theorem \ref{multivariable} is an extension of the stationary phase method. 

\begin{cor} 
\label{A_E_6_E_8} 
Assume that $\lambda > 0$ and $a \in \mathcal{A}^{\tau}_{\delta}(\mathbb{R}^{3})$. 
Then the following hold: 
\begin{enumerate} 
\item[(i)] 
For any $N \in \mathbb{N}$ such that $N>k+1$, as $\lambda \to \infty$, 
\begin{align} 
&Os\text{-} \int_{\mathbb{R}^{3}} e^{i\lambda (\pm x_{1}^{k+1}+x_{2}^{2}+x_{3}^{2})} a(x) dx 
= \sum_{\alpha \in \Omega_{1}} c_{\alpha} \lambda^{-\frac{\alpha_{1}+1}{k+1}-\frac{\alpha_{2}+1}{2}-\frac{\alpha_{3}+1}{2}} + O \left( \lambda^{-\frac{N-k}{k+1}} \right), 
\notag 
\end{align} 
\item[(ii)] 
For any $N \in \mathbb{N}$ such that $N>4$, as $\lambda \to \infty$, 
\begin{align} 
&Os\text{-} \int_{\mathbb{R}^{3}} e^{i\lambda (\pm x_{1}^{4}+x_{2}^{3}+x_{3}^{2})} a(x) dx 
= \sum_{\alpha \in \Omega_{2}} c_{\alpha} \lambda^{-\frac{\alpha_{1}+1}{4}-\frac{\alpha_{2}+1}{3}-\frac{\alpha_{3}+1}{2}} + O \left( \lambda^{-\frac{N-3}{4}} \right), 
\notag 
\end{align} 
\item[(iii)] 
For any $N \in \mathbb{N}$ such that $N>5$, as $\lambda \to \infty$, 
\begin{align} 
&Os\text{-} \int_{\mathbb{R}^{3}} e^{i\lambda (x_{1}^{5}+x_{2}^{3}+x_{3}^{2})} a(x) dx 
= \sum_{\alpha \in \Omega_{3}} c_{\alpha} \lambda^{-\frac{\alpha_{1}+1}{5}-\frac{\alpha_{2}+1}{3}-\frac{\alpha_{3}+1}{2}} + O \left( \lambda^{-\frac{N-4}{5}} \right), 
\notag 
\end{align} 
\end{enumerate} 
where $\Omega_{j}$ for $j=1,2,3$, and $c_{\alpha}$ with $c^{\pm}_{j,\alpha_{j}}$ are defined by \eqref{Omega_k} when $N_{1}=N$, and \eqref{c_alpha} with \eqref{c_pm_j_j_alpha_j}. 
\end{cor} 

\begin{proof} 
Set $n=3$ and $(m_{1},m_{2},m_{3};\pm_{1},\pm_{2},\pm_{3})=(k+1,2,2;\pm,+,+)$\\$=(4,3,2;\pm,+,+)=(5,3,2;+,+,+)$ on Theorem \ref{multivariable} respectively. 
\end{proof}

\end{document}